\documentclass[12pt,twoside]{amsart}
\usepackage{amsmath}
\usepackage{amssymb}
\usepackage{amscd}
\usepackage{xypic}
\xyoption{all}
\title[Variation of Hodge Structure and Hodge modules]{Variation of Hodge Structure and Hodge modules}
\author{Mohammad Reza Rahmati\\ }
\thanks{}
\address{
\hfill\break 
\hfill\break \\
\hfill\break }
\email{mrahmati@cimat.mx}

\newcommand{\comments}[1]{}

\def \Q{{\mathbb Q}}

\newtheorem{theorem}{Theorem}[section]
\newtheorem{proposition}[theorem]{Proposition}
\newtheorem{corollary}[theorem]{Corollary}
\newtheorem{lemma}[theorem]{Lemma}
\newtheorem{definition}[theorem]{Definition}
\newtheorem{remark}[theorem]{Remark}
\newtheorem{example}[theorem]{Example}



\begin{document}

\begin{abstract}
This text is an expository survey on the interplay between polarized variation of Hodge structure (PVHS) and the formalism of Hodge modules. 
We specifically review the extensions of a PVMHS over their singularities and its relation to mixed Hodge modules. 
\end{abstract}

\maketitle

%\noindent Version: 

%%%%%%%%%%%%%%%%%%%%%%%%%%%%%%%%%%%%%%%%%%%%%%%%%%%%%%%%%%%%%%%%%%%%
\section*{Introduction}
%%%%%%%%%%%%%%%%%%%%%%%%%%%%%%%%%%%%%%%%%%%%%%%%%%%%%%%%%%%%%%%%%%%%

\vspace{0.5cm}

The asymptotic behaviour of a Hodge structure was first studied by W. Schmid, \cite{SCH} and J. Steenbrink, \cite{ST1} to define 
limit mixed Hodge structure.  Classically, MHS's appear as solutions systems to regular holonomic flat $D$-modules, or equivalently a perverse sheaf. This is the content of Riemann-Hilbert correspondence. 

\begin{theorem}(Riemann-Hilbert (RH) correspondence) \cite{PS} There exists a functorial correspondence

\[ \mathcal{R}Hom_D(M,\mathcal{O}_X):D_{rh}^b(X) \to D^b(X,\mathbb{C}) \] 

is an equivalence of categories. 
\end{theorem}

A sheaf in $D^b(X,\mathbb{C})$ is called \textit{perverse} if it is isomorphic to $\mathcal{R}\text{Hom}_D(M,\mathcal{O}_X)$ or the solution module of some regular holonomic $M$. Suppose $M$ is a $D_X$-module. The sheaf $\text{Hom}_D(M,\mathcal{O}_X)$ is called the solution module of $M$. The derived functors $\mathcal{R}\text{Hom}_D(M,\mathcal{O}_X)$ are called higher solution module of $M$.

A $D$-module on an algebraic manifold $X$ is nothing other than a sheaf of $\mathcal{O}_X$-module with an extra $\mathbb{C}$-linear connection satisfying the Leibnitz rule. It would be the same to regard them as usual sheaves on cotangent bundle of $X$. The (proper)  pull-back and push-forward of $D$-modules are defined via correspondences. This roughly speaking tells, it is the same as usual sheaf theory on cotangent bundles. The category of $D$-modules is equipped with the basic (Grothendieck 6-functor) sheaf theoretic operations on $D$-modules. 

The $D$-modules under consideration in the text have geometric origins from Hodge theory, that is they are doubly filtered $(W,F)$, they are quasi-unipotent and the three filtration $W,F$ and the Malgrange-Kashiwara $V$-filtration are compatible. We will consider the extensions of MHS's as the corresponding solution to the extended $D$-module, which is defined by gluing of vector bundles with connection. 

Mixed Hodge modules are defined as the extensions of pure Hodge modules. 
A Hodge module always underlies a polarized variation of Hodge structure \textit{defined on a Zariski dense open subset} of the ambient space. In the theory of perverse sheaves, extensions along open strata of stratifications of complex manifolds is being done by the Deligne nearby functor, denoted $\Psi$ (notations $\psi_t,\ \psi_f$ are also used), by requiring a compatibility identity via restrictions.
The extension over a closed subvariety is explained by complexes of vector spaces via nearby and vanishing cycles maps. It is based on the question, how to describe the restriction of a vector bundle $M$ to a closed $i:D \hookrightarrow X$ and an open complement $ \ j:U \hookrightarrow X$, such that the original vector bundle becomes a pull back of the gluing of the two. 

In a more modern language, it is described as a $t$-structure on the derived categories of mixed Hodge modules and perverse sheaves, given by the exact triangles in the derived category of perverse sheaves

\[ i^*i_* \rightarrow 1 \rightarrow i^!i_* , \qquad j^*j_* \rightarrow 1 \rightarrow j^*j_! \]

\vspace{0.3cm}

\noindent
Here $j_!$ and $i_!$ are extensions by $0$, $j^! $ is restriction and $i^!$ means sections supported in the closed subset. The above isomorphisms define distinguished exact triangles in the derived category of perverse sheaves that explain the gluing process mutually. 

The nearby and vanishing cycle functors fit in the important triangle

\[ i^* \to \psi_f \to \phi_f \stackrel{-1}{\rightarrow} \]

in the derived category of perverse sheaves, where the first map is induced by adjunction, $f$ gives a local equation of degeneracy locus. The important fact is that these two functors do carry perverse sheaves to perverse sheaves (proved by Deligne). The associated long exact cohomology sequence becomes

\[ .. \to H^i(\psi_f) \to H^i(\phi_f) \to H^{i+1}(B \cap X_0) \to...\]

It follows that vanishing cycles are homology classes that are killed in $H^i(\psi_f)$, via the specialization (contraction) map $X_t \to X_0$. Moreover, the nearby functor would have a decomposition as $ \psi_f=\psi_f^{un} \oplus \psi_f^{\ne 1} ,\ \phi_f=\phi_f^{un} \oplus \phi_f^{\ne 1} $ where by any choice of a generator $T \in \pi_1(\Delta^*)$, $1-T$ is nilpotent on $\psi_f^{un}$. We have the following well-known exact triangles,

\[ i^*j_* \to \psi_f^{un} \stackrel{1-T}{\longrightarrow} \psi_f^{un} \to \ \ , \qquad i^* \to \psi_f^{un}j^* \stackrel{1-T}{\longrightarrow} \phi_f^{un} \to \]

such that $v \circ u=1-T$. There are homomorphisms $u$ and $v$ as 

\[ \psi_fM \stackrel{u}{\rightarrow} \phi_fM \stackrel{v}{\rightarrow} \psi_fM, \qquad v \circ u=(N=\log T_{u}) \otimes -1/2 \pi i \] 

The $D$-modules we will consider would have origins in Hodge theory, namely mixed Hodge modules, then they would automatically be regular holonomic, and we will assume they are also quasi-unipotent. Holonomicity of a filtered $D$-module $(M,F)$, means that $F$ is a good filtration of $M$ and 

\vspace{0.2cm}

\[ \dim \sqrt{\text{ann}_{Gr_F(D_X)}Gr_FM}=\dim(X) \]

which is the minimum number may be attained. The zero set of the ideal under the square defines a sub-variety of the cotangent bundle of $X$, namely characteristic variety. The above equation says this variety is a Lagrangian sub-variety of the cotangent bundle of $X$, i.e the symplectic form of $T^*X$ vanishes on this subvariety. If $X$ is smooth, a MHM on $X$, is always regular holonomic and is given by a 4-tuple $(M,F,K,W)$, where $W$ describes both of the weight filtration of $M$ and $K=\text{rat}(M)$. Then, a morphism is a pair of morphisms compatible with $\text{rat}$ and filtrations. A basic example is given by cohomology along fibers in a local fibration $f:X \to \Delta$ with $D=f^{-1}(0)$ a normal crossing divisor. It leads to the following diagram 

\begin{equation}
\begin{CD}
X_{\infty} @>>> U @>>> X @<<< E \\
@Vf_{\infty}VV @VVfV @VVfV @VVV\\
H @>e>> \Delta^* @>>> \Delta @<<< 0 
\end{CD}
\end{equation}

namely Specialization diagram. $X_{\infty}=X \times_{\Delta^*} H$ is called the canonical fiber. The Riemann-Hilbert correspondence guarantees the desired mixed Hodge modules as the filtered Guass-Manin system.

\section{Hodge Modules}

Let $X$ be a complex algebraic variety. When $X$ is smooth, then a mixed Hodge module on $X$ determines a 4-tuple $(M,F,K,W)$ where $M$ is a holonomic $D$-module with a \textit{good} filtration $F$ and, with rational structure $\text{DR}(M) \cong \mathbb{C} \otimes K \in \text{Perv}(\mathbb{C}_X)$, for a perverse sheaf $K$, and $W$ is a pair of weight filtrations on $M$ and $K$ compatible with $\text{rat}$ functor. $\text{DR}$ denotes the \textit{de Rham functor} shifted by the $\dim(X)$. The de Rham functor is dual to the solution functor. Denote by $MHM(X)$, the abelian category of Mixed Hodge Modules on $X$. $MHM(X)$ is equipped with a forgetful functor 

\[ \text{rat}:MHM(X) \to \text{Perv}(\mathbb{Q}_X) \], 

which assigns the underlying perverse sheaf/$\mathbb{Q}$. Sometimes the above concepts would be understood as elements in $D^bMHM(X)$ and $D_c^b(\mathbb{Q}_X)$ respectively, and the same for the functor $\text{rat}$. \textit{If $X=pt$, Then, $MHM(pt)$ is exactly all the polarizable mixed Hodge structures.}

A MHM has a weight filtration $W$, and we say it is \textit{pure of weight $n$}, if $Gr_k^W=0$ for $k \ne n$. Normally, the filtration $W$ is involved with a nilpotent operator on $M$ or the underlying variation of a mixed Hodge structure. 

A mixed Hodge modules (def.) is obtained by successive extensions of pure one. If the support of a pure Hodge module as a sheaf is irreducible such that no sub or quotient module has smaller support, then we say the module has \textit{strict support}. Any pure Hodge module will have a unique decomposition into pure modules with different strict supports, known as \textit{Decomposition Theorem}. A pure Hodge module is also called \textit{polarizable HM}. $MH_Z(X,n)^p$ will denote the category of pure Hodge modules with strict support $Z$. 

\textit{An $M \in HM_Z(X,n)$ determines a polarizable variation of Hodge structure. The converse of this fact is also true, that variation of Hodge structures determine a MHM}, \cite{SAI2}. Thus;

\begin{equation}
MH_Z(X,n)^p \simeq VHS_{gen}(Z,n-\dim Z)^p
\end{equation}

The right side means polarizable variations of Hodge structure of weight $n-\dim Z$ defined on a non-empty smooth sub-variety of $Z$.

The standard operations on the categories of sheaves can also be defined for $MHM(X)$. Here we have two additional operations namely Deligne nearby functor $\psi_f$ and the vanishing cycle functor $\phi_f$ along the fibers of $f \in \Gamma(X, \mathcal{O}_X)$, which fit into an exact triangle,

\begin{center} 
$i^{-1} \to \psi_f \stackrel{\text{can}}{\rightarrow} \phi_f \stackrel{[1]}{\rightarrow}... $\\
$\psi_f F=i^{-1}Rj_*j^{-1} F, \qquad \phi_f F=Cone(i^{-1}F \to \psi_f F) $
\end{center}
 
where $i:X_0 \hookrightarrow X$, $j:X\setminus X_0 \hookrightarrow X$. The above distinguished triangle can be considered as the definition of $\phi_f$, namely vanishing functor along $f$. The vanishing cycle functor is the mapping cone of the adjunction morphism $i^{-1}F \to \Psi_f F$. Thus we have a diagram 

\begin{equation}
\begin{CD}
i^*F @>>> \psi_*F @>can>> \phi_*F @>>>i^*F[-1] \\
@VVV @VVT-IV @VVvarV @VVV \\
0 @>>> \psi_*F @>=>>\psi_*F @>>>0
\end{CD}
\end{equation}

Assume $X$ is a local complete intersection and $\dim(X)=n+1$. Then $\mathbb{Q}_X[n+1]$ is a perverse sheaf. Denote $\psi_f\mathbb{Q}_X, \ \phi_f\mathbb{Q}_X$, be the nearby and vanishing cycle complexes on $X_0=f^{-1}(0)$. It is known that $\psi_f\mathbb{Q}_X[n], \ \phi_f\mathbb{Q}_X[n]$ are perverse. Then 
\[ \psi_{f,\lambda}\mathbb{Q}_X=\ker(T_s-\lambda), \qquad \phi_{f,1}\mathbb{Q}_X=\ker(T_s-id) \]

and $\phi_{f,\lambda}=\psi_{f,\lambda}$ for $\lambda \ne 1$. We know that

\[ H^j(F_x,\mathbb{Q})_{\lambda}:=H^j(\psi_{f,\lambda}\mathbb{Q}_X), \qquad \tilde{H}^j(F_x,\mathbb{Q})_{\lambda}:=H^j(\phi_{f,\lambda}\mathbb{Q}_X) \]

there is a short exact sequence

\[ 0 \to \tilde{H}^n(F_x, \mathbb{Q}) \to H^n(L_x,\phi_f \mathbb{Q}_X) \to K_x \to 0 \]

where $L_x$ is the link and $K_x$ is the kernel of the natural morphism, 

\[ \beta_{\phi}:H_c^n(F_x, \mathbb{Q})(-1) \to H^n(F_x, \mathbb{Q}) , \]

In the above $\beta_{\phi}$ is simply induced from the natural map $i^! \to i^*$. The reduced cohomology $\tilde{H}^j(F_x,\mathbb{Q})$ is sometimes refereed as vanishing cohomology. The sheaf $\phi_f\mathbb{Q}$ introduced by Deligne is a sheaf supported on $X_0$ whose cohomology calculates the vanishing cohomology. In the isolated hypersurface singularity case we have 

\begin{center}
$\tilde{H}^j(F_x,\mathbb{Q})=0, \qquad j<n$ 
\end{center}

that is equivalent to the perversity. The above short exact sequence can be interpreted as a relation between the cohomology of the milnor fiber and that of the link of singularity, \cite{DS}. The relation with monodromy is reflected in the Wang sequence

\[ \to H^j(L_x \setminus X_0) \to H^j(F_x)_1 \stackrel{N}{\rightarrow} H^j(F_x)_1(1) \to H^{j+1}(L_x \setminus X_0) \to ... \]

When $U$ is the complement of a normal crossing divisor $D \subset X$, $\mathcal{V}$ a local system on $U$ underlying a polarized pure Hodge structure of weight $n$, say $V$; such that the local monodromies around $U$ are quasi-unipotent, then there exists a unique Hodge module $V_X^{Hdg}$ of weight $(n+\dim X)$ having strict support $X$ and restricting to $V^{Hdg}$. The intersection complex 

\begin{center}
$IC_X^{\bullet} \mathcal{V}=j_{!_*} \mathcal{V}, \qquad j_{*!}:=\text{image}(j_! \to j_*)$ 
\end{center}

is the unique perverse extension of $\mathcal{V}[d]$ with strict support $X$. Here $j_{*!}:=\text{image}(j_! \to j_*)$ is the intermediate extension originally belonged to Deligne. Therefore, 

\[ V_X^{Hdg}=j_{!_*} \mathcal{V} \]  

The functor $\Psi_f$ defined before is a  special case of $j_{!*}$. Specifically, $j_{!*}$ is the result of applying $\Psi$ inductively along open strata of a startified manifold. 

On a stratified pseudo-manifold $X$ of $\dim(X)=n$, the intersection complex may be defined inductively, along the strata, starting from a constant sheaf $\mathbb{R}$, Using the Deligne extension $\Psi$ we described above. The resulting complex is the intersection complex

\[ \text{IC}_X^{\bullet}=\tau_{\leq \bar{p}(n)-n}Ri_{n*}...\tau_{\leq \bar{p}(2)-n}Ri_{2*}\mathbb{R}_{X-\Sigma}[n] \]

where $\bar{p}=\{\bar{p}(2),...,\bar{p}(n)\}$ is the perversity, $\tau$ is trunctation of the complex, and $i_k:U_k \hookrightarrow U_{k+1}, \ U_k=X-X_{n-k}, \ X-\Sigma=U_2$, \cite{B}. 

\begin{remark} 
In a simple extension of the local system $\mathcal{H} \to \Delta^*$ associated to the Milnor fiberation of $f:\mathbb{C}^{n+1} \to \mathbb{C}$, we have 

\[ j_{!_*}\mathcal{H} := \{\  \displaystyle{\sum_{\alpha,l}f^{\alpha}\exp(\dfrac{-N}{2 \pi i}\log f)m_{\alpha,l} \ } \} \]
\end{remark}

\begin{theorem}\cite{AR}
Let $U$ be the complement of a normal crossing divisor in a compact Kahler manifold $X$. Then intersection cohomology with coefficient in a polarized VHS on $U$ is isomorphic to $L^2$ cohomology for a suitable complete Kahler metric on $U$. 
\end{theorem}

It follows that $L^2$ cohomology is finite dimensional, and also intersection cohomology carries a pure Hodge structure. The above theorem also gives a decomposition theorem for the direct image $f_*IC_X^{\bullet}L$, with $L$ a local system on $U$ and $f$ a proper or projective morphism.

\section{V-filtration}

The Kashiwara-Malgrange $V$-filtration of a regular holonomic $D_X$-module associated to a subvariety $Y \hookrightarrow X$ is an increasing filtration generally indexed by $\mathbb{Q}$ satisfying simple axiomatic conditions which characterize it. We explain this by an example. Let $X=\mathbb{C}$ with coordinate $t$ and $Y=0$. Fix a rational number $r \in (-1,0)$, and let $M=\mathcal{O}_{\mathbb{C}}[t^{-1}]t^r$, with $\partial_t$ acting on the left in the usual way. For each $\alpha \in \mathbb{Q}$ define $V_{\alpha}M \subset M$ to be the $\mathbb{C}$-span of $\{t^{n+r}|n \in \mathbb{Z}, n+r >-\alpha \}$. The following properties are easy to check 

\begin{itemize}
\item The filtration is exhaustive and left continuous: $\cup V_{\alpha}M=M$, and $V_{\alpha+\epsilon}=V_{\alpha}M$, for $0 <\epsilon <<1$
\item Each $V_{\alpha}M$ is stable under $t^i\partial_t^j$ if $i >j$.
\item $\partial_tV_{\alpha}M \subset V_{\alpha+1}M$, and $t.V_{\alpha}M \subset V_{\alpha-1}$.
\item The associated graded 

\[ Gr_{\alpha}^VM=V_{\alpha}/V_{\alpha-\epsilon}= \begin{cases}
\mathbb{C}t^{-\alpha} \ \text{if} \alpha \in r + \mathbb{Z}\\
0 \qquad \text{otherwise} 
\end{cases}
\]
is an eigen-space of $t\partial_t$ with eigenvalue $-\alpha$.
\end{itemize}

The last item implies that the set of indices that $V_{\alpha}M$, jumps is discrete. The above construction may be generalized to define $V$-filtration for a regular holonomic $D$-module on $X$ that are quasi-unipotent along a closed sub-variety $Y$. If $Y$ is smooth, then for such a module, there always exists a unique filtration satisfying similar properties as listed above, called the $V$-filtration along $Y$. Then $t$ would be replaced by the ideal sheaf of $Y \hookrightarrow X$. In case $Y$ is not smooth this construction can be done using embedding by graph. For instance, if $f:X \to \mathbb{C}$ is a holomorphic function, and and let $\imath_f :X \to X \times \mathbb{C}=Y$ be the inclusion by graph. Let $t$ be the coordinate on $\mathbb{C}$, and let 

\[ V_{\alpha}\imath_*M=D_{x \times 0}\langle t^i \partial_t^j|i-j >-|\alpha| \rangle \]

for $\alpha \in \mathbb{Q}$, \cite{AR}. Let $X_0=f^{-1}(0)$ be possibly a singular fiber. A holonomic $D_X$-module $M$ has quasi-unipotent monodromy along $X_0$, if the monodromy action on $\psi_t (DR \ M)$ is quasi-unipotent. Any regular holonomic $D_X$-module with quasi-unipotent monodromy is specializable along $X_0$, i.e the module can be extended over $X_0$. This can be done using embedding by graph of $f$, namely
$i_f:X \to X \times \mathbb{C}$. In fact the module 

\[ \tilde{M}=i_+M =\displaystyle{\int_{i_f}}(M,F) = M[\partial_t], \qquad DR \ i_+M=i_*M \]

works out here,

\begin{equation}
DR_{X \times 0} Gr_V^{\alpha}\tilde{M} \cong \begin{cases}
\psi_{t,\lambda} DR_X M[-1] \qquad -1 \leq \alpha <0,\\
\phi_{t,\lambda} DR_X M[-1] \qquad -1 < \alpha \leq 0.
\end{cases}
\end{equation}

The $V$-filtration is indexed by $\mathbb{Q}$ such that $t\partial_t-\alpha$ is nilpotent on $Gr_V^{\alpha}$, and

\begin{center}
$t:F_pV^{\alpha} \tilde{M} \to F_pV^{\alpha +1} \tilde{M}, \qquad \alpha >1 $\\[0.2cm]
$\partial_t:F_p Gr_V^{\alpha} \tilde{M} \to F_p Gr_V^{\alpha -1} \tilde{M} , \qquad \alpha>0$
\end{center}

are isomorphisms. By definition,

\begin{equation}
\psi_f(M)=\displaystyle{\bigoplus_{-1<\alpha \leq 0}}Gr_V^{\alpha}(\tilde{M}), \qquad  
\phi_f(M)=\displaystyle{\bigoplus_{-1<\alpha < 0}}Gr_V^{\alpha}(\tilde{M}) \oplus Gr_V^{-1}(\tilde{M})
\end{equation}

\begin{equation}
DR \psi_fM=\psi_f DR\  M[-1], \qquad DR\phi_fM=\phi_f DR \ M[-1]
\end{equation}

Moreover;

\begin{equation}
F_p\tilde{M}=\displaystyle{\sum_{i}}\partial_t^i(V^{-1}\tilde{M} \cap j_*j^{-1}F_{p-i}\tilde{M})
\end{equation}

where $j:X \times \mathbb{C}^* \to X \times \mathbb{C}$. This means that the $V$-filtration together with the Hodge filtration on the complement of $f^{-1}(0)$ determine the total Hodge filtration $F$, \cite{SAI2}. 

\section{Mixed Hodge Modules}

Roughly speaking, a mixed Hodge module is obtained by extension of polarized pure Hodge modules. A mixed Hodge module on complex algebraic manifold $X$ is given by an open cover $\{X_i\}$ of $X$, $U_i=X_i-Y_i, Y_i=t_i^{-1}(0), t_i:X \to \mathbb{C}$ and gluing data $(M \vert U_i,M\vert Y_i, u_i, v_i)$, \cite{PS}. The phenomenon is a method of gluing vector bundles. The gluing data is $(M_U,M_D,u,v)$ satisfying $v \circ u=N \otimes -1/2 \pi i$. It may be presented in the diagram,

\[ (\psi M=M_U) \stackrel{u}{\rightarrow} (\phi M=M_D) \stackrel{v}{\rightarrow} (M_U=\psi M) \]

where $\psi$ is the Deligne extension (of course the sequence is not exact). 

\begin{theorem} \cite{LI} The category of regular holonomic $D_X$-modules is the same as the category of diagrams $\displaystyle{M \substack{v \\ {\ \rightleftarrows \ } \\ u} N } $ of vector spaces, where $1_M-uv$ and $1_N-vu$ are invertible.
\end{theorem}

\begin{proof}\cite{BEI}
For a vector space $V$ and $\phi \in End(V)$, let $(V,\phi)^0$ be the maximal subspace on which $\phi$ acts in a nilpotent way. Consider the category $C$ of diagrams $(V_0^{\prime},V_1^{\prime},\phi,u,v)$, where $V_0^{\prime},V_1^{\prime}$ are vector spaces, $\phi \in Aut V_1^{\prime}$, and $(V_1^{\prime}, id_{V_1^{\prime}}-\phi)^0 \displaystyle{ \substack{v \\ {\ \rightleftarrows \ } \\ u} V_0^{\prime} } $ are such that $v \circ u=id-\phi$. Then we have the following equivalence 

\begin{equation}
(\displaystyle{V_0 \substack{v \\ {\ \rightleftarrows \ } \\ u} V_1 }) \mapsto ((V_0^{\prime},u \circ v)^0,V_1^{\prime},id_{V_1}-(v \circ u),u,v)
\end{equation}

\end{proof}

The category regular holonomic $D$-modules on $(\Delta^*,0)$ is isomorphic to the above category, since 
modules over $0$ are only vector spaces, and over $\Delta^*$ are vector spaces with monodromy. Under this identification $\Psi_f(V,T)=(V,id_V-T)^0$ where $\Psi_f$ is the unipotent extention functor. One should encode the aforementioned gluing method as a standard way for two vector spaces where one is equipped with an isomorphism namely monodromy. Although one can do this task in many other ways, but the one explained will preserve perversity and also other expected properties via pull back and pushforward in this category.  

The category of perverse sheaves on the disk $D$ which are locally constant on $D^*$ is equivalent to the category of quivers of the form $ \psi \ \substack{c\\ \leftrightarrows \\v} \ \phi $ i.e. finite dimensional vector spaces $\psi, \phi$ with maps as indicated. 
A quiver $ \psi \ \substack{c \\ \rightleftarrows \\v} \ \phi $ corresponds to $j_*L[1]$ where $L=\psi$ and $T=I+v \circ c$. Then 

\[ \phi=\text{image}(c) \oplus \ker(v) \]

The extension of a Hodge module over a normal crossing compactification may be explained as follows. Assume $i:U \hookrightarrow X$ is the open inclusion and $X-U=D$, a normal crossing divisor. A MHM on $X$ determines (in a unique way) two MHM's , $M$ on $U$ and $M'$ on $D$ with gluing morphisms $u:\psi^{un}M \to M'$ and $v:M' \to \psi^{un}M(-1)$ such that $vu=N$, where $\psi^{un}$ is the uni-potent ($\lambda = 1$) part of $\psi$, \cite{SAI2}, \cite{SAI5}. Then, it is easily verified that 

\[ M^{\prime}=\text{Im} (u) \oplus \ker(v) \]

and $u $ and $v$ induce morphisms

\[ u:(M,W) \to (M^{\prime},W[1]), \qquad v:(M^{\prime},W) \to (M,W[1]) \]

The converse is also true. Given the above filtered maps then $N=uv=vu$ is nilpotent, and $W$ is the monodromy filtration for $M^{\prime}$. One can show that $u,v$ will preserve the weight and relative monodromy filtrations. 

A mixed Hodge modules corresponds to a variations of mixed Hodge structure. This means that $M$ is endowed with an increasing filtration $W$, called weight filtration such that $Gr_i^WM$'s are polarized Hodge modules of weight $i$. Here the extension can not be arbitrary. The imposed conditions are, 

\begin{itemize}
\item The original mixed Hodge module is polarized.
\item The relative weight filtration that is the weight filtration on $Gr_k^W M$, associated to the induced nilpotent operator $Gr_k N$ exists. 
\item The Hodge filtration extends over Deligne extension.
\item The nearby and vanishing cycle functors are well defined for $M$.
\item The filtrations $F,W, V^{(i)}, \ (0 \leq i \leq n)$ are compatible, where $V^{(i)}$ are the are the Kashiwara-Malgrange filtrations along the coordinate hyperplanes.
\end{itemize}

\noindent
These conditions together are called \textit{admissibility conditions}, \cite{SAI5}. It sould be understood as a criteria in order to make the (minimal) extension possible. Then the underlying perverse sheaves or local system will satisfy similar conditions via the functor $\text{rat}$. Through all the text we have assumed this condition, although not stated specifically. 

\begin{theorem} \cite{SAI2}, \cite{SAI5}
Admissible variation of mixed Hodge structures are mixed Hodge modules.
\end{theorem}

The extension of mixed Hodge modules should be understood as gluing two vector bundles via an isomorphism. The key point is that one of the vector bundles or its fiber is equipped with a monodromy. Another condition is to  distinguish the pure and the mixed case. In the mixed case one needs to add the admissibility criteria to be able to extend the structure. To extend the underlying  variation of mixed Hodge structure, it suffices to extend the corresponding $D$-module via the Riemann-Hilbert correspondence. Then Theorem 3.2 guarantees that the construction extends also the Hodge and the weight filtrations, and is unique. In some special cases the gluing could be described easier for the structure of the Gauss-Manin system. 

\begin{example}
The Gauss-Manin system $\mathcal{G}: =R^{n} f_{*} \mathbb{C}_{X'}$ of a polynomial or holomorphic map $f$ with isolated singularity, with the Milnor representative $f:X' \to T^{\prime}$ is a module over the ring $\mathbb{C}[\tau,\tau^{-1}]$, where $\tau$ is a variable, and comes equipped with a connection, that we view as a $\mathbb{C}$-linear morphism $\partial_{\tau}:\mathcal{G} \to \mathcal{G}$ satisfying Leibnitz rule. In order to extend it as a rank $\mu $-vector bundle on $\mathbb{P}^1$, one is led to study lattices i.e. $\mathbb{C}[\tau], $ and $\mathbb{C}[t]$-submodules which are free of rank $\mu$. In the chart $t$, the Brieskorn lattice 

\begin{center}
$\mathcal{G}_0=image (\Omega^{n+1}[\tau^{-1}] \to \mathcal{G})=\dfrac{\Omega^{n+1}[t]}{(td-df \wedge)\Omega^{n+1}[t]}$
\end{center}

is a free $\mathbb{C}[t]$ module of rank $\mu$. It is stable by the action of $\partial_{\tau}=-t^2\partial_t$. Therefore $\partial_t$ is a connection on $\mathcal{G}$ with a pole of order $2$. We consider the increasing exhaustive filtration $\mathcal{G}_p:=\tau^p\mathcal{G}_0$ of $\mathcal{G}$. In the chart $\tau$, there are various natural lattices indexed by $\mathbb{Q}$, we denote them by $V^{\alpha}$, with $V^{\alpha-1}=\tau V^{\alpha}$. On the quotient space $C_{\alpha}=V^{\alpha}/V^{>\alpha}$ there exists a nilpotent endomorphism $(\tau\partial_{\tau}-\alpha)$. The space $\oplus_{\alpha \in [0,1[}C_{\alpha}$ is isomorphic to $H^n(X_{\infty},{\mathbb{C}})$, and $\oplus_{\alpha \in [0,1[}F^pC_{\alpha}$  is the limit MHS on $H^n(X_{\infty},{\mathbb{C}})$. A basic isomorphism can be constructed as 

\begin{center}
$\begin{CD}
\dfrac{\mathcal{G}_p\cap V^{\alpha}}{\mathcal{G}_{p-1}\cap V^{\alpha}+\mathcal{G}_p \cap V^{>\alpha}}=Gr_F^{n-p}(C_{\alpha}) \\   
@V\tau^pV{\cong}V \\
\dfrac{ V^{\alpha+p} \cap \mathcal{G}_0}{ V^{\alpha}\cap \mathcal{G}_{-1}+ V^{>\alpha} \cap \mathcal{G}_0}=Gr_{\alpha+p}^V(\mathcal{G}_0/\mathcal{G}_{-1})
\end{CD}$.
\end{center}

and the gluing is done via this isomorphisms. 
\end{example}

\section{Local Systems Over $\mathbb{C}^*$}

We give a simple explanation of unipotent local systems on $\mathbb{C}^*$. This provides a picture of general unipotent local systems of arbitrary $D$-modules. We study the local system of vector spaces over $\mathbb{C}$ of dimension $n$ with a unipotent monodromy given by 

\[
   M^{un}=
  \left[ {\begin{array}{cccc}
   1 & -1 &  &  \\
   0 & 1 & -1 &  \\
   \vdots  &   &  \ddots  & -1 \\
         &    &         & 1
  
  \end{array} } \right]
\]

which has a filtration of length $n$. Lets begin by putting
 
\[ \mathcal{J}^{(n)} :=\ \displaystyle{\sum_{k=0}^{n-1}} \mathcal{O}_U .\log^k \]

They satisfy a system of inclusions and projections in an obvious way as; 

\[ 0 \ \substack{{\hookrightarrow} \\ {\twoheadleftarrow}} \ \mathcal{J}^{(1)} \ \substack{{\hookrightarrow} \\ {\twoheadleftarrow}} \ ... \ \substack{{\hookrightarrow} \\ {\twoheadleftarrow}} \ \mathcal{J}^{(0)} \  ... \]

Set,

\begin{center}

$\mathcal{J}_f^{a,b}:= \mathcal{J}^a/\mathcal{J}^b $

\end{center} 
Then,

\[ \mathcal{J}^{0,1} \hookrightarrow \mathcal{J}^{0,2} \hookrightarrow  ...\hookrightarrow \mathcal{J}^{0,3} ... , \qquad \mathcal{J}_f^{0,\infty}=\lim_{\leftarrow}\mathcal{J}^{0,b} \]

\[ \mathcal{J}^{-1,0} \hookrightarrow \mathcal{J}^{-2,0} \hookrightarrow  ...\hookrightarrow \mathcal{J}^{-3,0} ..., \qquad \mathcal{J}_f^{0,\infty}=\lim_{\rightarrow}\mathcal{J}^{a,0} \]

\begin{definition} 

\[\mathcal{J}_f^{-\infty,\infty}= \lim_{\leftarrow}\lim_{\rightarrow}\mathcal{J}^{a,b}, \qquad M^{-\infty,\infty}=\lim_{\leftarrow}\lim_{\rightarrow}(M \otimes_{\mathcal{O}_{\mathbb{C}^*}} \mathcal{J}^{a,b}) \]

where $M$ is any $D$-module on $\mathbb{C}^*$.
\end{definition}

\noindent
We have $\mathbb{D}\mathcal{J}^{a,b}=Hom_{\mathcal{O}_{\mathbb{C}^*}}(\mathcal{J}^{a,b},\mathcal{O}_{\mathbb{C}^*}) \cong \mathcal{J}^{-b,-a}$, by

\[ \mathcal{J}^{a,b} \otimes \mathcal{J}^{-b,-a} \to \mathcal{J}^{0,1}=\mathcal{O}_{\mathbb{C}^*}, \qquad \langle f(s), g(s) \rangle=Res_{s=0}f(s)g(-s)ds \]

Therefore, 
 
\begin{center}
$\mathbb{D}J^{0,n}=J^{-n,0} \cong J^{0,n} , \qquad \mathbb{D}(M \otimes_{\mathcal{O}_U} \mathcal{J}_f^{a,b})=\mathbb{D}M \otimes_{\mathcal{O}_U} \mathcal{J}_f^{-b,-a} $
\end{center}

If we have a non-degenerate bilinear (or polarization) pairing 

\[ K:M \otimes M \to \mathcal{O}_{\mathbb{C}^*} \]

then the induced bilinear form, 

\begin{center}
$ \tilde{K}: \psi_{\lambda}M^{-\infty,\infty} \otimes \psi_{\lambda}M^{-\infty,\infty} \to \mathbb{C} $
\end{center} 
\begin{equation} \langle m \otimes f(s), n \otimes g(s) \rangle=Res_{s=\alpha}f(s)g(-s) K(m,n) ds 
\end{equation}

is non-degenerate. Later we use this as an strategy in order to extend the polarization in a unipotent extension. The combinatorial framework of the $D$-modules $\mathcal{J}^{a,b}$ allows to explain the duality on Deligne extensions in a simple way, via the trace map and residue. 

\begin{lemma} (Key lemma)

\[ j_*M^{-\infty,\infty}=j_!M^{-\infty,\infty} \]

and similarly for $M_k^{-\infty,\infty}$. 

\end{lemma}

The point is, if 

\begin{center}
$\alpha: j_!(M \otimes f^*J^{0,n}) \to j_*(M \otimes f^*J^{0,n})$ 

\end{center}

be the natural map, then 
$\ker \alpha \hookrightarrow \Psi_f^{un} $ (the unipotent extension functor) and this injection is an equality for $ n>>0$.  This idea originally belongs to Beilinson, in the paper "How to glue Perverse sheaves". The above lemma plays the main task to explain the minimal extension in the unipotent case. For a filtered ring $A$ with $Gr A=\oplus A^i/A^{i+1}=\oplus \mathbb{Z}(i)$ and can define 

\[ \langle.,.\rangle:A \times A \to \mathbb{Z}(i) ,\qquad \langle f,g\rangle=Res_{\tilde{t}=1}(f.g^-d\log \tilde{t}) \]

where $g \to g^-$ is a natural involution on $A$. then he sets 

\[ A^1=(A^{-1})^{\perp}, \qquad A^{a,b}:=A^a/A^b \cong Hom(A^{-b}/A^{-a},\mathbb{Z}(-1)) \]

This ring in many ways is like the local system $\mathcal{J}$. The same holds if replace $\mathbb{Z}(i)$ with $(\mathbb{Z}/l^n)(i)$ for the $l$-adic local systems by repeating every thing word by word and, 

\[ A_{et}=\lim_{\rightarrow a} \lim_{\leftarrow b}A^{a,b}
\]

in this case. 

\section{Polarization}

The duality of $D$-modules is the duality of vector bundles with connections. In this way it would be a type of Serre duality of coherent sheaves. As a first step is better we stress that the vector bundle is filtered by a holomorphic filtration $F$. In order to reflect the Hodge structure and polarization one is led to consider the graded structure associated to this filtration. Let $(\mathcal{G}, \nabla, F, \mathcal{H}_{\mathbb{Q}},S)$ be a polarized variation of Hodge structure of weight $n$. The flat connection $\nabla $ makes the vector bundle $\mathcal{G}$ into a left $D$-module. Now consider a polarization

\begin{center}
$S:\mathcal{H}_{\mathbb{Q}} \times \mathcal{H}_{\mathbb{Q}} \to \mathbb{Q}(-n)$ 
\end{center}

of the variation. By definition we have $S(F^p,F^q)=0$ for $p+q >n$. Thus $S$ descends to a non-degenerate pairing between $Gr_F^{k}\mathcal{G}$ and $Gr_F^{-n+k}\mathcal{G}$, for all $k$. Thus, we get an isomorphism

\begin{equation}
\displaystyle{\bigoplus_{k \in \mathbb{Z}}Gr_F^k\mathcal{G} \to \bigoplus_{k \in \mathbb{Z}}Hom_{\mathcal{O}_X}(Gr_F^{n-k} \mathcal{G}, \mathcal{O}_X)}
\end{equation}

Moreover, we obtain that, $S$ is flat with respect to the Gauss-Manin connection, and 

\[ dS(\lambda_1,\lambda_2)=S(\nabla \lambda_1,\lambda_2)+S(\lambda_1,\nabla \lambda_2) \]

\begin{definition}
If $\mathcal{G}=\oplus \mathcal{G}_k$ is a graded module, then its graded dual is defined by

\[ \mathcal{G}^{\vee}=\oplus Hom_{\mathcal{O}_X}(\mathcal{G}_{-k},\mathcal{O}_X) \]

the $\ i$-th derived functor of $D$ is evidently

\[ \mathcal{G}^{\vee i}=\oplus Ext_{\mathcal{O}_X}^i(\mathcal{G}_{-k},\mathcal{O}_X) \]

\end{definition}

\begin{definition}

The (Verdier) dual of a $D$-module is defined by

\[ \mathbb{D}M=\text{Ext}_{D}^d(M,D \otimes_{\mathcal{O}} \omega_X^{-1}) \]

A polarization of $M$ is an isomorphism 

\[ M \cong \mathbb{D}(M)(-w) \]
\end{definition}

\noindent
A polarization of a Hodge module is a duality $\mathbb{D}M=M(n)$ where $(n)$ is the Tate twist, which is essentially defined by the shift of the complex by $n$.

\begin{theorem}\cite{SCH} Let $\mathcal{G}$ be the Hodge module associated to a polarized variation of Hodge structure $(\mathcal{H}_{\mathbb{Q}}, \nabla, F, S)$ of weight $n$, with $S:\mathcal{H}_{\mathbb{Q}} \otimes \mathcal{H}_{\mathbb{Q}} \to \mathbb{Q}(-n)$ the polarization. Then, we have the isomorphism 

\begin{equation}
\displaystyle{\bigoplus_{k \in \mathbb{Z}}Gr_F^k\mathcal{G} \to \bigoplus_{k \in \mathbb{Z}}Hom_{\mathcal{O}_X}(Gr_F^{n-k} \mathcal{G}, \mathcal{O}_X)}
\end{equation}

given by (up to a sign factor) $\lambda \to S(\lambda,-)$, for $\lambda \in Gr_F^k\mathcal{G}$. 
\end{theorem}

If $f:X \to Y$ is a projective morphism between smooth complex varieties, and $M$ a (pure) Hodge module on $X$ with strict support and of weight $n$, then $R^kf_*M$ is a Hodge module on $Y$ of weight $n+k$. If $M \in MH(X)^p$, then its cohomology carries a Hodge structure. The Lefschetz decomposition theorem may be stated as;

\[ PGr_i^WM=\ker(GrN^{i+1} :Gr_i^WM \to Gr_{-i-2}^WM) \]
\[ \sum Gr^W N^m=\oplus PGr_{i+2m}^WM \cong Gr_{i}^WM \]

The duality functor is stable under $PGr^W\psi_f$ in case the $D$-module is defined via the fibration by $f$. If two graded module or vector-space having a Lefschetz decomposition relative to a specific nilpotent operators of degree 1. Then, a bilinear or hermitian form will polarize them if and only if the level graded polarizations dualize the corresponding primitive sub-spaces. Moreover, the two corresponding bilinear forms would be isomorphic if and only if the set of graded polarizations are isomorphic.

\begin{theorem} \cite{SAI3} Assume $f:X \to Y$ is a morphism of smooth analytic manifolds, and $(M,F,K;W)$ is a mixed Hodge module polarized by, namely $S$. Then

\[ (-1)^{j(j-1)/2}R^jf_*S \circ(id \otimes l^j):P_lR^{j}f_*K \otimes P_lR^{j}f_*K \to \mathbb{Q} \]

is a polarization on the primitive components, for $j \geq 0$.
\end{theorem}

Polarization plays an important role in the extension process we are concerned and actually is equivalent to that. Consider the the nondegenerate form

\[ K:\mathcal{H}^{\prime} \otimes_{\mathcal{O}} \overline{\mathcal{H}} \to \mathbb{C}[t,t^{-1}] \] 

It induces an isomorphism 

\[ \mathcal{H}^{\prime \vee} \cong_{\mathcal{O}} \overline{\mathcal{H}} \]

We can glue the above bundles by this isomorphism obtained from the polarization. Therefore, in situation of (1) we have

\[ \mathcal{H}^{(0) \vee} \cong \overline{\mathcal{G}_{\infty}} ,\qquad  \Rightarrow \qquad \Omega_f^{\vee} \cong \overline{H^n(X_{\infty},\mathbb{C})} \]

as PVMHS, and PMHS respectively. The corresponding connections are given by

\[ \nabla': \mathcal{H}^{\prime} \to \dfrac{1}{z}\Omega^1 \otimes \mathcal{H}^{\prime} , \qquad \nabla: \overline{\mathcal{H}} \to {z}\Omega^1 \otimes \overline{\mathcal{H}} \]

respectively, \cite{DW} exp. 1, pages 12, 13.

\section{Isolated hypersurface singularities}

Asuume $f:\mathbb{C}^{n+1} \to \mathbb{C}$ is a germ of isolated singularity. Then the variation of mixed Hodge structure $(H^n(X_{\infty},\mathbb{C}), F_{\lim}, W(N))$  associated to the cohomology of the Milnor fibers is polarized namely with $S$. In this case the $V$-filtration and the Brieskorn lattice $H''$ which are vector bundles of rank $\mu$ (the Milnor number of $f$) can be defined. Moreover the polarization of the Gauss-Manin system is given by the K. Saito higher residue pairing $P_S$ which a flat bilinear form on the disc. 

\begin{proposition} (\cite{H1} prop. 5.1)
Assume $\{(\alpha_i,d_i)\}$ is the spectrum of a germ of isolated singularity $f:\mathbb{C}^{n+1} \to \mathbb{C}$. There exists elements $s_i \in C^{\alpha_i}$ with the properties

\begin{itemize}

\item[(1)] $s_1,...,s_{\mu}$ project onto a $\mathbb{C}$-basis of $\bigoplus_{-1<\alpha<n} Gr_V^{\alpha} H''/Gr_V^{\alpha} \partial_t^{-1}H''$. 

\item[(2)] $s_{\mu+1}:=0$; there exists a map $\nu:\{1,...,\mu\} \to \{1,...,\mu,\mu+1 \}$ with $(t-(\alpha_i+1)\partial_t^{-1})s_i =s_{\nu(i)}$

\item[(3)] There exists an involution $\kappa:\{1,...,\mu\} \to \{1,...,\mu \}$ with $\kappa=\mu+1-i$ if $\alpha_i \ne \frac{1}{2}(n-1)$ and $\kappa(i)=\mu+1-i$ or $\kappa(i)=i$ if $\alpha_i = \frac{1}{2}(n-1)$, and  

\[ P_S(s_i,s_j)= \pm \delta_{(\mu+1-i)j}.\partial_t^{-1-n} \].

\end{itemize}
\end{proposition}

\noindent
The basis discussed in 6.1 is usually called a good basis. The condition (1) correspond to the notion of opposite filtrations. Two filtrations $F$ and $U$ on $\mathcal{G}$ are called opposite (cf. \cite{SAI6} sec. 3) if 

\[ Gr_p^FGr_U^q \mathcal{G}=0 , \qquad \text{for} \ p \ne q \]

In our situation this amounts to a choice of a section $s:H''/\partial_t^{-1}H'' \to H''$ of the projection $pr:H'' \to H''/\partial_t^{-1}H''$ and $Image(s)$ generates $\oplus_{\alpha} (H'' \cap C^{\alpha})$.

\noindent
$V^{\alpha} H''$ is the submodule generated by $s(V^{\alpha}\Omega_f)$.
 
\begin{proposition} (\cite{SAI6} prop. 3.5)
The filtration 

\[ U^pC^{\alpha}:=C^{\alpha} \cap V^{\alpha+p}H'' \]

is opposite to the filtration Hodge filtration $F$.
\end{proposition}

The proof of the theorem 6.1 is based on construction of a $\mathbb{C}$-linear isomorphism, 

\begin{center}
$\Phi:H^n(X_{\infty},\mathbb{C}) \to \Omega_f\cong \Omega^{n+1}/df \wedge \Omega^n$
\end{center}

The mixed Hodge structure on $\Omega_f$ is defined by using the isomorphism $\Phi$. This means that 
\begin{center}
$W_k(\Omega_f)=\Phi W_kH^n(X_{\infty},\mathbb{Q}),  
 \qquad F^p(\Omega_f)= \Phi F^pH^n(X_{\infty},\mathbb{C})$ 
\end{center}

and all the data of the Steenbrink MHS on $H^n(X_{\infty},\mathbb{C})$ such as the $\mathbb{Q}$ or $\mathbb{R}$-structure is transformed via the isomorphism $\Phi$ to that of $\Omega_f$. Specifically; in this way we also obtain a conjugation map

\begin{equation}
\bar{.}:\Omega_{f,\mathbb{Q}} \otimes \mathbb{C} \to \Omega_{f,\mathbb{Q}} \otimes \mathbb{C}, \qquad \Omega_{f,\mathbb{Q}}:=\Phi^{-1}H^n(X_{\infty},\mathbb{Q})
\end{equation}

defined from the conjugation on $H^n(X_{\infty},\mathbb{C})$ via this isomorphism. 

\begin{theorem}  \cite{R} 
Assume $f:(\mathbb{C}^{n+1},0) \to (\mathbb{C},0)$, is a holomorphic germ with isolated singularity at $0$. Then, the isomorphism $\Phi$ makes the following diagram commutative up to a complex constant;

\begin{equation}
\begin{CD}
\widehat{Res}_{f,0}:\Omega_f \times \Omega_f @>>> \mathbb{C}\\
@VV(\Phi^{-1},\Phi^{-1})V                   @VV \times *V \\
S:H^n(X_{\infty}) \times H^n(X_{\infty}) @>>> \mathbb{C}
\end{CD} \qquad \qquad  * \ne 0
\end{equation}

\vspace{0.3cm} 

\noindent
where, 

\[ \widehat{Res}_{f,0}=\text{res}_{f,0}\ (\bullet,\tilde{C}\ \bullet) \]

\vspace{0.3cm} 

\noindent
$\text{res}_{f,0}$ is the Grothendieck residue associated to $f$, and $\tilde{C}$ is defined relative to the Deligne decomposition of $\Omega_f$, via the isomorphism $\Phi$. If $J^{p,q}=\Phi^{-1} I^{p,q}$ is the corresponding subspace of $\Omega_f$, then

\begin{equation}
\Omega_f=\displaystyle{\bigoplus_{p,q}}J^{p,q} \qquad \tilde{C}|_{J^{p,q}}=(-1)^{p} 
\end{equation}

In other words;
\begin{equation}
S(\Phi^{-1}(\omega),\Phi^{-1}(\eta))= * \times \ \text{res}_{f,0}(\omega,\tilde{C}.\eta), \qquad 0 \ne * \in \mathbb{C}
\end{equation}
\end{theorem}

\begin{remark} (\cite{PH} page 37)
Setting
\begin{center}
$\psi_s^i (\omega,\tau)=\displaystyle{\int_{\Gamma(i)}e^{-\tau f}\omega}$\\[0,2cm]
$\bar{\psi}_s^i(\omega^{\prime},\tau)=\displaystyle{\int_{\Gamma^{\prime}(i)}e^{+\tau f}\omega^{\prime}}$
\end{center}

with $\zeta=\frac{\omega}{df},\ \zeta^{\prime}=\frac{\omega^{\prime}}{df}$, the expression (which is the same as in the proof)

\begin{equation}
\mathcal{K}_s([\zeta],[\zeta^{\prime}])(\tau)= \displaystyle{\sum_{i=1}^{\mu}}\psi_s^i(\tau,\omega)\bar{\psi}_s^i(\tau,\omega^{\prime})=\displaystyle{\sum_{r=0}^{\infty}}\mathcal{K}_s^r([\zeta],[\zeta^{\prime}])(\tau).\tau^{-n-r}
\end{equation}

is a presentation of K. Saito higher residue pairing.

\end{remark}

\begin{corollary}
The polarization $S$ of $H^n(X_{\infty})$ will always define a polarization of $\Omega_f$, via the isomorphism $\Phi$. In other words $S$ is also a polarization in the extension.
\end{corollary}

The Riemann-Hodge bilinear relations for the MHS on $\Omega_f$ and its polarization $\widehat{Res}$ would be that of an opposite MHS to $(H^n(X_{\infty}),S)$. 

\begin{corollary} (Riemann-Hodge bilinear relations for $\Omega_f$)
Assume $f:\mathbb{C}^{n+1} \to \mathbb{C}$ is a holomorphic germ with isolated singularity. Suppose $\mathfrak{f}$ is the corresponding map to $N$ on $H^n(X_{\infty})$, via the isomorphism $\Phi$. Define 

\[ P_l=PGr_l^W:=\ker(\mathfrak{f}^{l+1}:Gr_l^W\Omega_f \to Gr_{-l-2}^W\Omega_f) \]

Going to $W$-graded pieces;

\begin{equation}
\widehat{Res}_l: PGr_l^W \Omega_f \otimes_{\mathbb{C}} PGr_l^W \Omega_f \to \mathbb{C}
\end{equation}

is non-degenerate and according to Lefschetz decomposition 

\[ Gr_l^W\Omega_f=\bigoplus_r \mathfrak{f}^r P_{l-2r} \]

\noindent
we will obtain a set of non-degenerate bilinear forms,

\begin{equation}
\widehat{Res}_l \circ (id \otimes \mathfrak{f}^l): P Gr_l^W \Omega_f  \otimes_{\mathbb{C}} P Gr_l^W \Omega_f  \to \mathbb{C}, 
\end{equation} 

\begin{equation}
\widehat{Res}_l=res_{f,0}\ (id \otimes \tilde{C} .\  \mathfrak{f}^l)
\end{equation}

Then, 

\begin{itemize}
\item $\widehat{Res}_l(x,y)=0, \qquad x \in P_r, \ y  \in P_s, \ r \ne s $

\item If $x \ne 0$ in $P_l$, 

\[ const \times res_{f,0}\ (C_lx,\tilde{C} .\  \mathfrak{f}^l .\bar{x})>0  \]

where $C_l$ is the corresponding Weil operator.

\end{itemize}

\end{corollary}

\begin{example}
For instance by taking $f=x^3+y^4$, then as basis for Jacobi ring, we choose 
\begin{center}
$z^i:\ 1, \ y, \ x, \ y^2, \ xy, \ xy^2$ 
\end{center}

They correspond to the top forms with degrees 

\begin{center}
$l(i): \ 7/12, \ 10/12, \ 11/12, \ 13/12, \ 14/12,\ 17/12$
\end{center} 

describing the pole (Hodge) filtration on the $\mathcal{H}:=R^nf_* \mathbb{C}$, respectively. The above basis projects onto a basis

\begin{center}
$\displaystyle{\bigoplus_{-1 < \alpha=l(i) -1 < n}Gr_{\alpha}^VH'' \twoheadrightarrow Gr_V \Omega_f}$ 
\end{center}

as in Theorem 6.1. The Hodge filtration is explained as follows. First, we have $h^{1,0}=h^{0,1}=3$. Therefore, because $\Phi$ is an isomorphism. 

\[ <1 .\omega,\ y. \omega, \ x .\omega > =\Omega_f^{0,1}, \qquad <y^2. \omega, \ xy.\omega  , \ xy^2.\omega > =\Omega_f^{1,0} \]

where $\omega =dx \wedge dy$, and the Hodge structure is pure, because $Gr_2^WH^n(X_{\infty})=0$.

\begin{center}
$\overline{<1. dx \wedge dy, \ y. dx \wedge dy, \ x. dx \wedge dy>}= $ \\ $<c_1. xy^2 .dx \wedge dy, \ xy .dx \wedge dy , \ y^2 .dx \wedge dy> $. 
\end{center}
\end{example}

\section{Hermitian duality}

In the following we explain a method of descent of duality for $D$-modules, originally belonged to C. Sabbah and M. Saito cf. \cite{SA4}, \cite{SAI3}. 
Assume $X=Z \times \mathbb{C}$, where $Z$ a complex manifold identified with $Z =Z \times 0$, and Let $M$ be a holonomic $D_X$-module. Define 

\begin{equation}
M_{\alpha,p}:=\displaystyle{\bigoplus_{k=0}^p M[t^{-1}] \otimes e_{\alpha,k}}
\end{equation}

with $e_{\alpha,k}=0$ for $k<0$ and $e_{\alpha,k}=t^{\alpha}(\log t)^k/k!$ otherwise. We have natural maps

\begin{center}
$... \leftrightarrows M_{\alpha,p} \leftrightarrows M_{\alpha,p+1} \leftrightarrows M_{\alpha,p+2} \to ...$
\end{center} 

where the composite of the two converse arrow is nilpotent. Then the maps,

\begin{center}
$Gr_{\alpha}^V M \to GR_{-1}^V M_{\alpha,p}, \qquad 
 m_0 \mapsto \bigoplus_{k=0}^p [-(\partial_t t +\alpha)]^k m_0 \otimes e_{\alpha,k}$\\[0.5cm]
$Gr_{-1}^V M_{\alpha,p} \to Gr_{\alpha}^V M, \qquad
\sum_{k=0}^p m_k \otimes e_{\alpha,k} \mapsto \sum_{k=0}^p [-(\partial_t t +\alpha)]^k m_{p-k}$
\end{center}

for $p$ large enough induce isomorphisms;

\begin{center}
$Coker(t \partial_t) \cong Gr_{\alpha}^VM  \cong \ker(t \partial_t)$
\end{center}

The limit is called moderate nearby cycle module, denoted $\psi_{t,\lambda}^{mod} M$, which plays the same role as $\ j_*(M \otimes f^*J^{0,n}), \ n>>0$ in section 4. The case of moderate vanishing cycle module $\phi_{t,1}^{mod}$ may be done in some what similar way, by considering the inductive system $M \to M_{-1,p}$ instead of the single module $M_{\alpha,p}$, and the action of $N$ is the endomorphism $-\partial_t. t$ on $Gr_0^V M$. Then we have,

\begin{center}
$Can=-\partial_t: Gr_{-1}^V M \leftrightarrows Gr_{0}^V M : t=Var $. 
\end{center}

which are isomorphism, \cite{SA4}. Let 

\begin{center}
$S:M \otimes M \to \mathbb{C}[[t,t^{-1}]]$ 
\end{center}
be a duality. It extends formally to 

\[ \psi_t S:\psi_tM \otimes \psi_tM \to Db_{\mathbb{C}}^{mod(0)}, \qquad \phi_t S:\phi_tM \otimes \phi_tM \to Db_{\mathbb{C}}^{mod(0)} \]

\noindent
by

\begin{equation}
\displaystyle{S(\sum_{k=0}^{p} \mu_k \otimes e_{\alpha,k}\ ,\ \sum_{l=0}^{p} m_l \otimes e_{\alpha,l})= \sum_{k+l=p} (\mu_k , m_l)}e_{\alpha,k}\overline{e_{\alpha,l}}
\end{equation}

$D_{\mathbb{C}}^{mod(0)}$ is the ring of $C^{\infty}$ distributions with moderate growth in dimension 1. These distributions naturally receive a doubly indexed $V$-filtration w.r.t the coordinates $t$ and $\bar{t}$. \cite{SA3}. 

\[ \sum_{\alpha,p}\mathbb{C}\{t\}[t^{-1}]\mathbb{C}\{\bar{t}\}[\bar{t}^{-1}](\log|t|)^p \] 

which is a $D_{\mathbb{C}} \otimes D_{\bar{\mathbb{C}}}$-module in the obvious way. Then, for $-1 \leq \alpha < 0$  we obtain the induced forms,

\begin{equation}
\psi_{\lambda} S: Gr_{\alpha}^V M \otimes_{\mathbb{C}} Gr_{\alpha}^V M \to \mathbb{C}, \qquad \phi_{1} S:Gr_{0}^V M\otimes_{\mathbb{C}} Gr_{0}^V M \to \mathbb{C}
\end{equation}

with properties;
\begin{center}
$\psi_{\lambda} S(N \bullet, \bullet)=\psi_{\lambda} S( \bullet,N \bullet), \qquad \phi_{1} S(N \bullet, \bullet)=\phi_{1} S( \bullet,N \bullet)$
\end{center}

which says $N$ is an infinitesimal isometry of the descendants. We also obtain a set of positive definite bilinear maps,

\begin{equation}
\psi_{\lambda,l} S \otimes (id \otimes N^l): P Gr_l^W Gr_{\alpha}^V M \otimes_{\mathbb{C}} P Gr_l^W Gr_{\alpha}^V M \to \mathbb{C}
\end{equation}

The form $S$ is non-degenerate in a neighbourhood of $Z$ iff all the forms $P \psi_{\lambda,l} S $ are non-degenerate. Similar statement is true for hermitian or polarization forms. The graded pairings $\psi_{\lambda}S, \ -1 \leq \alpha <0 $ are given by the formal residue of the form $S$ at $t=\alpha$ and $t=0$ respectively for $\psi_{\lambda}S$ and $\phi_1S$. We have proved the following.

\begin{theorem}
Assume $(\mathcal{G},F,W,H,S)$ is a polarized MHM (hence regular holonomic) with quasi-unipotent underlying variation of mixed Hodge structure $H$, defined on a Zariski dense open subset $U=X \setminus Z$ of an algebraic manifold $X$, where $Z$ is a smooth projective hypersurface. Then, the Gauss-Manin system $\mathcal{G}$ has a smooth extension to all of $X$ and the extended MHM is also \textit{polarized}. The polarization on the fibers can be described by residues of the Mellin transform of a formal extension of the polarization $S$ over the elementary sections, by the two formulas 

\[ \displaystyle{\psi_{\lambda} S\langle \sum_{l=0}^{p} m_l \otimes e_{\alpha,l}, \overline{\sum_{l=0}^{p} m_l \otimes e_{\alpha,l}} \rangle = * . \ Res_{s=\alpha} \langle \tilde{S} , |t|^{2s} dt \wedge d\bar{t} \rangle } ,\qquad  * \ne 0, \ \alpha \ne 0 \]
\[ \phi_1 S(\bullet, \bar{\bullet})= * . \ Res_{t=-1} \langle \tilde{S} , |t|^{2s} \mathcal{F}_{loc}dt \wedge d\bar{t}  \rangle , \qquad * \ne 0\].

\end{theorem}

A polarization of mixed Hodge modules also defines a conjugation map $C_X$ (or similar logarithmic one $C_X^{mod(Z)}$). This functor has the following natural properties,

\[ \psi_{t,\lambda}^{\text{mod}} \circ C_X \cong C_Z \circ \psi_{t,\lambda}^{\text{mod}}, \qquad \lambda \ne 1 \]

\[ \phi_{t,1}^{\text{mod}} \circ C_X \cong C_Z \circ \psi_{t,1}^{\text{mod}}, \qquad \lambda=1 \]

Both of the isomorphisms commute with the nilpotent operator $N$ and are also compatible with the gluing data for regular holonomic $D$-modules. This is a special case of the Kashiwara conjugation functor, \cite{BK} where we have the following commutative diagram.

\begin{equation}
\begin{CD}
MHM(X)_{\lambda} @>{\psi_{t,\lambda}^{\text{mod}} \circ C_X}>> MHM(X)_{-\lambda}\\
@V{DR_{X,\lambda}}VV                   @VV{DR_{X,-\lambda}}V \\
\text{Perv}_{-\lambda}(\mathbb{C}) @>>{\psi_{t,\lambda}^{\text{mod}} \circ c_X^{-1}}> \text{Perv}_{\lambda}(\mathbb{C})
\end{CD}
\end{equation}

\section{Fourier-Laplace Transform}

Another interesting view of extensions of PVMHS is described by Fourier-Laplace transform of sheaves. For the set up we consider $M(*\infty)=M \otimes D_{\mathbb{P}^1}(*\infty)$ and define its Fourier-Laplace transform 

\[ \widehat{M}:=q_+(p^+M(*\infty)) \otimes \mathcal{E}^{-t\tau}) , \qquad \mathcal{E}^{-t\tau}= (\mathcal{O}_{\mathbb{P}^1 \times \mathbb{C}}, \nabla=d-\tau dt-td\tau) \] 

The Fourier-Laplace transform can also be equivalently defined by; 

\[ \widehat{M}=\text{coker}( \mathbb{C}[\tau] \otimes M \stackrel{\nabla_t-\tau dt}{\longrightarrow} \mathbb{C}[\tau] \otimes M ) ,\qquad \tau .m  :=\partial_t .m \]

This also applies to polarization of $D$-modules. If we have a polarization,

\[ K:\mathcal{H}^{\prime} \otimes_{\mathcal{O}} \overline{\mathcal{H}} \to \mathcal{L}^{\mathbb{R}-an} \] 

Then it carries over 

\[ \widehat{K}:\widehat{\mathcal{H}^{\prime}} \otimes_{\mathcal{O}} \imath^{+} \overline{\widehat{\mathcal{H}}} \to \mathcal{L}^{\mathbb{R}-an} \] 

where $\mathcal{L}^{\mathbb{R}-an}$ is set of distributions as in section 7 (Here $\imath:\mathbb{P}^1=\mathbb{C} \cap \infty \to \mathbb{P}^1$ is $ z \mapsto -z$ and $\imath^+$ is necessary for we use $\exp(\overline{t \tau})$ not $\exp(-\overline{t \tau})$ ). In a way that the distribution on the the integral is twisted by $\exp(-\overline{t\tau}). \exp(t \tau)$. The product after Fourier transform is 

\[ (\sum \tau^im_i)dt \otimes (\overline{\sum \tau^in_i)dt} \mapsto [\psi \to \sum_{i,j} k(m_i,n_j)\tau^i \bar{\tau}^j e^{-\overline{t\tau}}.e^{t \tau}\psi dt \wedge d\bar{t} ] \] 

up to a complex constant, \cite{SA4}.

\begin{example} \cite{SA4} 

\begin{itemize}

\item $M=\mathbb{C}[t]\langle\partial_t \rangle / (t-c) \Longrightarrow K(m,\bar{m})=\delta_c, \ \widehat{K}(m,\bar{m})=i/2 \pi \exp(\overline{c\tau}-c\tau)$  
\item $M=\mathbb{C}[t]\langle\partial_t \rangle / (t\partial_t-\alpha) \Rightarrow K(m,\bar{m})=|t|^{2 \alpha}, \ \widehat{K}(m,\bar{m})=\Gamma(\alpha+1)/\Gamma(-\alpha) |\tau|^{-2(\alpha+1)}$

\end{itemize}
\end{example}

\begin{theorem}\cite{DW}
Assume $\mathcal{H}'=R^nf_*\mathbb{C}_{X'}$ be the local system associated to a holomorphic isolated singularity $f$. Consider the map 

\[ F:\Omega_X^{n+1} \to i_* \bigcup_z Hom (H_n(X,f^{-1}(\eta.\dfrac{z}{|z|}),\mathbb{Z}) \cong \oplus_i \mathbb{Z}\Gamma_i, \mathbb{C}) \] 

\[ \omega \mapsto [z \to (\Gamma_i \to \int_{\tilde{\Gamma}_i}e^{-t/z}\omega)],  \]

and define 

\[ \mathcal{H}:=Im(F) \]

where $\Gamma_i$ are the classes of  Lefschetz thimbles, and $\tilde{\Gamma}_i$ is the extension to infinity. Then the vector bundle $\mathcal{H}$ is exactly the Fourier-Laplace transform of the cohomology bundle $R^nf_*\mathbb{C}_{X'}=\cup_tH^n(X_t,\mathbb{C})$, equipped with a connection with poles of order at most two at $\infty$.

\[ (\mathcal{H}', \nabla') \leftrightarrows (\mathcal{H}, \nabla) \]

\end{theorem}

\begin{corollary}
In case of the PVMHS associated to the Milnor fibration of an isolated hypersurface singularity $f$, the modified Grothendieck residue

\[ \widehat{Res}_{f,0}=res_{f,0}(\bullet,\hat{C}\bullet) \]

where $\hat{C}$ is defined relative to the Deligne-Hodge decomposition of $\Omega_f$ as before, is the Fourier-Laplace transform of the polarization $S$ on $H^n(X_{\infty},\mathbb{C})$, that is

\[ \widehat{Res}=\ * . \ ^FS , \qquad * \ne 0 \]

\end{corollary}

The corollary is also a consequence of uniqueness of polarization. 

\section{Normal crossing compactification}

We present a brief of the work of M. Saito in \cite{SAI5}. We explain nearby and vanishing cycles inductively through a stratification of the normal crossing divisor. As the question is local, we may assume, $ X=\Delta^n, \ D_i=\{x_i=0\}, \ D_I=\cap_{i \in I} D_i $. Suppose $M$ is a regular holonomic $D_X$-module with quasi-unipotent monodromy along $D_i$'s. Then $M$ is given by, $E_I^{\nu}, \ \nu \in (\mathbb{C}/\mathbb{Z})^n $, on the hyper-cover obtained by the simplicial structure of $D$, equipped with the morphisms;

\begin{center}
$\text{can}_i:E_I^{\nu} \to E_{I \cup i}^{\nu}, \qquad \text{Var}_i:E_I^{\nu} \to E_{I \setminus i}^{\nu}$
\end{center}

such that $\ \text{can}_i \circ \text{Var}_i=\text{Var}_i \circ \text{can}_i=N_i:E_I^{\nu} \to E_{I}^{\nu}$. We shall assume the sheaves $E_I^{\nu} $ are given as 

\begin{equation}
E_I^{\nu}=\Psi_{x_1}^{\nu_1}...\Psi_{x_n}^{\nu_n}\mathcal{L}, \qquad \Psi_{x_i}^{\nu_i}=\begin{cases} \psi_{x_i}^{\nu_i}\mathcal{L}[-1] \qquad i \notin I \\
\phi_{x_i}^{0}\mathcal{L}[-1]   \qquad \text{otherwise}.
\end{cases}
\end{equation}

and define;

\begin{center}
$M^{\alpha}=\cap_i(\cup_i \ker((x_i\partial_i-\alpha_i)^j:M \to M))$
\end{center}

Thus,
 
\[
E_I^{\nu}=M^{\alpha+1_I}, \ (\nu \equiv \alpha \mod \mathbb{Z}^n ,\ \alpha \in \mathbb{C}^n), \qquad 1_I=(...,1,...) \] 

Actually, $\nu \in (\mathbb{Q} \cap [-1,0))^{n+1}$ would be a set of exponents of different monodromies. Then, $\ \text{can}_i=\partial_i, \ \text{Var}_i=x_i, \ N_i=x_i\partial_i-\alpha$, and

\begin{center}
$\psi_{x_i}=\ker(T_{j,s}-e(\alpha))$, 
\end{center}

with the same for $\phi$. Then the Kashiwara-Malgrange $V$-filtration is by definition, 

\[ V_{\beta}^{(i)}=M \cap \prod_{\alpha \leq \beta} M^{\alpha} \].

Suppose now $D$ is defined by a single equation, $g=x^m=x_1^{m_1}...x_n^{m_n}, \ m \in \mathbb{N}^n$, set $\ N_J=\prod_{i \in J} N_i, \ \text{can}_J=\prod_{i \in J} \text{can}_i, \ \text{Var}_J=\prod_{i \in J} \text{Var}_i$. Set 

\[ N=\log(T_u) \] 

to be the logarithm of monodromy on the punctured disc, normalized by twisting with $(n)$. Then the specialization of the system is given by

\begin{equation}
\tilde{E}_I^{\nu}= \begin{cases} \text{coker} \{ \prod_{i \in {I \cap\bar{m}}}(N_i -m_iN)=\tilde{N}_I \}, \qquad 0 \notin I  \\  
\text{coker} \{ \left( 
\begin{array}{cc}
\{(\prod_{i \in {I \cap\bar{m}}}(N_i -m_iN)-N_{I \cap \bar{m}}\}N^{-1}  &  -Var_{I \cap \bar{m}}\\
can_{I \cap \bar{m}} &   N
\end{array} \right) =\tilde{N}_I \} 
\end{cases}
\end{equation} 

where the morphisms are injective endomorphisms of $E_{I\setminus\bar{m}}^{\nu^{\prime}}[\log(T_u)]$, and $E_{I\setminus\bar{m}}^{\nu^{\prime}+\nu_0m}[N=\log(T_u)] \oplus E_{I^{\prime}}^{\nu^{\prime}}[N=\log(T_u)], \ I^{\prime}=I \setminus 0$, respectively. We continue with, 

\begin{center}
$\ \Psi^{n+1}(\Psi_g \mathcal{F}) \cong \{\tilde{E}_I^{\nu},\ \tilde{\text{can}}_i,\ \tilde{\text{Var}}_i,\ \tilde{N}_i \}$ , 
\end{center}

\begin{proposition} \cite{SAI5}
Let $((H,F,W),\ N_i;\ S)$ be a PVMHS of weight $w$, where $W$ is the monodromy filtration for $\ \sum N_i\ $ shifted by $w$. We take $T$ as the monodromy on the disc, and set $s=\log(T_u), \ l =|I|$, and 

\begin{center}
$(\tilde{H}_I;F,W)=\text{coker}(\tilde{N}_I:(H[s];F[l],W[-2l]) \to (H[s];F,W)) $\\[0.5cm]
$  F^p(H[s])=\sum_jF^{p+j}H \otimes s^j, \qquad W_k(H[s])=\sum_jW_{k+2j}H \otimes s^j  $. 
\end{center}

Then 

\begin{center}
$\ ((\tilde{H}_I,\ F,\ W),\ s=\log(T_u),\ N_i+m_i.s;\ \tilde{S}_I) \ $, 

\end{center}

extends the original PVMHS over $D_I$ and is of weight $\ w+l-1\ $, where $W$ is the monodromy filtration for $\ s+\sum(\tilde{N}_i+m_is)\ $ shifted by $\ w+l-1\ $ and $\tilde{S}_I$ is defined by 

\[ \tilde{S}_I(\tilde{u},\tilde{v})=Res \ S(\tilde{N}_I^{-1} \tilde{u}, \tilde{v}) \]

In a way that $S$ is extended to 

\begin{center}
$\ S:H[s,s^{-1}] \otimes H[s,s^{-1}] \to \mathbb{C}[s,s^{-1}]\ $, 

\end{center}

by $\ S(u \times s^i, v \otimes s^j)=(-1)^iS(u,v) \otimes s^{i+j}\ $ , and $\ Res \ ( \sum a_i \otimes s^i)=a_{-1} $.

\end{proposition}

We sketch the idea of proof from the appendix in \cite{SAI5}, due to M. Kashiwara. One may assume $\ m_i=1 \ $. By definition we have

\[ \tilde{H}=H[s]/\text{coker}(\prod_{i \in I}(s-N_i)) \cong \oplus_{0 \leq j <1}H \otimes s^j \]
\[ \tilde{S}(s^j.u, s^k.v)=S(u,(-1)^jRes_{s=0}(\prod_{i \in I}(s-N_i)^{-1})s^{j+k}v) \]

where $H$ is identified with $H \otimes 1 (\subset H[s])$ and $(s-N_i)^{-1}=s^{-1} \sum_{j \geq 0} N_i^{-1}s^{-j}$. The proof will proceed by induction on $l=|I|$ and $\dim(H)$, and the assertion is clear for $\dim(H)=0$. It would also be clear for $l=1$, for then $H \simeq \tilde{H}$. Then the proof of theorem may be understood to prove an induction criteria using identities

\[ S^{\prime}(\text{can} \otimes id)=S(id \otimes \text{Var}) \]

by uniqueness. The morphisms $\text{can}, \  \text{Var}$ extend by $\text{can} \otimes 1, \  \text{Var} \otimes 1$ etc. Moreover, if you consider the formal structure $(\mathbb{R}[N^{\prime}]/(N^{\prime l}), N^{\prime}, S^{\prime})$ of weight $1-l$ with 
$S^{\prime}(N^{\prime i},N^{\prime j})=(-1)^i , \ \text{if} \ i+j=l-1, \text{and} \ 0 \ \text{otherwise}$ Then we have 

\[ (\tilde{H},\tilde{S})=(H,S) \otimes (\mathbb{R}[N^{\prime}]/(N^{\prime l}),S^{\prime}) \]

\begin{theorem}\cite{SAI5} For a reduced irreducible separated complex analytic space $X$ of dimension $n$, we have an equivalence of categories,
\[ MH_X(X,w) \cong VHS_{gen}(X,w-n)^p \]
where the right hand side is the inductive limit of $VHS(U,w)^p$ the category of polarizable variations of Hodge structures of weight $w$ with quasi-unipotent local monodromies on smooth dense Zariski open subsets $U$. Moreover, the polarizations correspond bijectively.

\end{theorem}

\section{Higher Residue pairing}

This section provide a concrete form of dulity for mixed Hodge modules, namely Higher residue pairing. The construction of higher residues and primitive forms originally belongs to K. Saito, \cite{S2}, \cite{LLS}.
It provides a standard method to describe a parametric family of dualities for polarized variation of mixed Hodge structures. However, conventionally the duality for $D$-modules is a non-degenerate hermitian sesqui-linear form. The method of K. Saito is to express a Serre duality between the Hodge sub-bundles $\mathcal{H}^{(-k)} \supset \mathcal{H}^{(-k-1)}$ of the Hodge filtration and the corresponding components of a co-filtration $ \hat{\mathcal{H}}^{(k)} \to \hat{\mathcal{H}}^{(k+1)}...  $ associated to the Gauss-Manin system. The method we explain it here is a some what different method explained in the second reference. It is based on the identification of the complexes 

\[ (PV(X)=\sum PV^{i,j}(X),\partial, \bar{\partial}) \qquad \leftrightarrows \qquad ( A(X)=\sum A^{i,j}(X),d, \bar{d}) \]
\[ \alpha_{I,J} \partial_Iz \otimes  \partial_J\bar{z},  \qquad \leftrightarrows \qquad \beta_{I,J}dz \wedge d\bar{z} \]

of smooth poly-vector fields on the left, with the space of smooth complex differential forms on $X$. It gives a filtered quasi-isomorphism 

\[ (PV(X)((t)), Q_f=\bar{\partial}_f+t\partial) \to (A(X)((t)), d+t^{-1}df \wedge \bullet) \]

where $Q_f$ is the corresponding coboundary to $d+t^{-1}df \wedge \bullet$ via a specific isomorphism 

\[ PV(X)((t)) \cong A(X)((t)) \] 

In fact setting 

\[ d_f^+:=d+\dfrac{df}{t}\wedge \ , \qquad d_f^-:=td+df \wedge \ \]

the maps 

\[ \Gamma^+:(PV(X)((t)), \Q_f) \cong ((A(X)((t)),d_f^+) , \qquad \Gamma^+:(PV(X)((t)), \Q_f) \cong ((A(X)((t)),d_f^-) \]

are filtered isomorphisms via $F^k PV(X)((t))=t^k.PV(X)[[t]]$, and similarly we may filter the other complex. The natural embedding 

\[ \imath:(PV_c(X)[[t]], Q_f) \hookrightarrow (PV(X)[[t]], Q_f) \]

where $c$ states for compact support, defines a quasi-isomorphism, and if we set

\[ \mathcal{H}_{(0)}^f:=H^*(PV(X)[[t]], Q_f) , \qquad \mathcal{H}^f=\mathcal{H}_{(0)}^f \otimes_{\mathbb{C}[[t]]} \mathbb{C}((t))\]

In fact, we have all the isomorphisms

\[ \mathcal{H}_{(-k)}^f=t^k\mathcal{H}_{(0)}^f=H^*(t^kPV(X)[[t]], Q_f) =H^*(t^k\Omega_X^*[[t]], d_f^-)=\mathcal{H}_f^{(-k)} \]

In this way we obtain a Hodge filtration 

\[ F^k\mathcal{H}_{(0)}^f=\mathcal{H}_{(-k)}^f , \qquad Gr_F^k\mathcal{H}_{(0)}^f=t^k Jac(f) \]

then the trace map

\[ Tr:PV_c(X) \to \mathbb{C} \]

provides a $\mathbb{C}[[t]]$-homomorphism $\widehat{Res}^f$ as 

\[ \mathcal{H}_{(0)}^f \longrightarrow \mathcal{O}_{S,0}[[t]] , \qquad \widehat{Res}^f=\sum_k \widehat{Res}_k^f(\bullet)t^k \]

with $\widehat{Res}_k^f$ the higher residues. Similarly, we obtain the higher residue pairing 

\[ K^f( \ , \ ):\mathcal{H}_{(0)}^f \times \mathcal{H}_{(0)}^f \to \mathcal{O}_{S,0}[[t]], \qquad  K^f( \ , 1 ):=\widehat{Res}^f \]

$\mathcal{H}_{(0)}^f$ will also inherits a connection as

\[ \nabla:\mathcal{H}_{(0)}^f \to t^{-1}.\mathcal{H}_{(0)}^f \otimes \Omega_{S,0}^1 \].

The higher residue $K^f$ defines a duality on $\mathcal{H}_{(0)}^f$. We can use the trace map

\[ PV_c(X)[[t]] \times PV_c(X)[[t]] \to \mathbb{C}[[t]] , \qquad (\alpha_1.v_1(t), \alpha_2.v_2(t)) \mapsto v_1(t)v_2(-t)Tr(\alpha_1,\alpha_2) \]

here the convention $\overline{\alpha.v(t)}=v(-t)\alpha$ is used. We equip $PV_c(X)((t))$ with the symplectic pairing 

\[ \omega(\alpha_1.v_1(t), \alpha_2.v_2(t))=Res_{t=0} v_1(t)v_2(-t)Tr(\alpha_1,\alpha_2) \]

If we have an admissible variation of mixed Hodge structure on a Zariski open subset underlying our MHM on $X \setminus f^{-1}(0)$, the limit Hodge filtration pairs with an opposite filtration $\Phi$ to define a complex variation of MHS. Here by complex we mean we forget about the real structures. In such a case we always can find a decomposition $\mathcal{H}^f=\mathcal{H}_{(0)}^f \oplus \mathcal{L}$, such that $t^{-1}\mathcal{L} \subset \mathcal{L}$. Then we have

\[ K^f(B,B) \subset \mathbb{C}, \qquad K^f(\mathcal{L}, \mathcal{L}) \subset t^{-2}\mathbb{C}[t^{-1}], \qquad \omega(\mathcal{L},\mathcal{L})=0 \]

\begin{theorem} \cite{S2}, \cite{LLS}
Let $s_1,s_2$ be local sections of $\mathcal{H}_{(0)}^f$.
\begin{itemize}
\item $K^f(s_1,s_2)=\overline{K^f(s_2,s_1)}$.
\item $K^f(v(t)s_1,s_2)=K^f(s_1,v(-t)s_2)=v(t)K^f(s_1,s_2)$, $v(t) \in \mathcal{O}_S[[t]]$.
\item $\partial_V.K^f(s_1,s_2)=K^f(\partial_Vs_1,s_2)+K^f(s_1,\partial_Vs_2)$, for any local section of $T_S$.
\item $(t\partial_t+n)K^f(s_1,s_2)=K^f(t\partial_t.s_2,s_1)+K^f(s_1,t \partial_t.s_2)$
\item The induced pairing on 
\[ \mathcal{H}_{(0)}^f/t.\mathcal{H}_{(0)}^f\otimes \mathcal{H}_{(0)}^f/t.\mathcal{H}_{(0)}^f \to \mathbb{C} \]

is the classical Grothendieck residue.

\end{itemize}
\end{theorem}

\section{Application to Neron models of PVHS}

Let $(\mathcal{H},F)$ be a variation of Hodge structure. We are interested to family of intermediate Jacobians 

\[ J(H_s)=H_{s,\mathbb{Z}} \setminus H_{s,\mathbb{C}}/F^pH_{s,\mathbb{C}}=\text{Ext}_{MHM}^1(\mathbb{Z},H_s) \]

\[ J(\mathcal{H})= \displaystyle{\bigcup_{s \in S^*}J(H_s)} \]

associated to such VMHS, called the Neron model of $\mathcal{H}$ (here we have assumed the weight is 2p-1). The sections of the bundle $J(\mathcal{H})$ are called Normal functions. Define 

\[ \text{NF}(S^*,\mathcal{H})_S^{\text{ad}}:=\text{Ext}^1(\mathbb{Z}_{S^*},\mathcal{H}) \]

called admissible normal functions, where $\text{Ext}$ is taken in the category of ${\text{VMHS}(S^*)_S^{\text{ad}}}$ the category of admissible variation of mixed Hodge structures, \cite{SAI8}.

To extend $J(\mathcal{H})$ to a space over $S$, we let $M$ be the polarized Hodge module on $S^*$, obtained from the variation $\mathcal{H}$. On $S^*$ we have an extension sequence 

\[ 0 \to \mathcal{H} \to \mathcal{J} \to \mathbb{Z} \to 0 \]

and therefore an extension

\[ 0 \to M \to N \to \mathbb{Q}_S^H[n] \to 0 \]

with $\mathbb{Q}_S^H[n]$ the trivial Hodge module 
of weight $n$ on $S$. Dualizing the extension and applying a Tate twist, we also have

\[ 0  \to \mathbb{Q}_S^H[n]  \to N_{\nu}^{\vee} \to M(-1) \to 0 \]

with $ N^{\vee}=\mathbb{D}(N)(-n)$.

\begin{example} \cite{SCH2}
Consider the trivial family of Elliptic curves $E \times \Delta^*$ where $E=\mathbb{C}/(\mathbb{Z}+\tau. \mathbb{Z})$, has an automorphism of order 6. Consider the trivial family $E \times \Delta^*$, as well as its quotient by $\mathbb{Z}/6. \mathbb{Z}$. We denote the local system corresponding to the quotient by 
$\mathcal{H}$. By choosing the bases, the monodromy on the cohomology takes the form;

\begin{center}
$  T=\left( 
\begin{array}{cc}
0  &  1\\
-1 &   1
\end{array} \right)$
\end{center}

with eigen-values $\tau$ and $\bar{\tau}$. The Deligne extension is given by $\mathcal{O}e_1 \oplus \mathcal{O}e_2$ with connection defined by

\[ \nabla e_1=-e_1 \otimes \dfrac{ds}{6s}, \qquad \nabla e_2=-e_2 \otimes \dfrac{5ds}{6s} \] 

Admissibility condition can be tested by pulling back along the branch cover $s=t^6$ to make the monodromy unipotent. If we only need to consider the family $E \times \Delta^*$, there is a map $g:\Delta \to \mathbb{C}$ such that,

\[ g(\tau.t)-\tau .g(t) \in \mathbb{Z}+\tau.\mathbb{Z} \]

because the normal function is pulled back from the original family, $g$ may be chosen so that $g(\tau .t)-\tau .g(t) =0$, this choice of $g$ represents the pull back of the extended normal function, its value over $0$ is $g(0)=0$. Thus the pull back of any admissible normal function over $E \times \Delta$ go through the origin. 

\end{example}

We have the commutative diagram of non-degenerate bilinear forms,

\begin{center}
$\begin{array}[c]{cccccc}
K: & M & \otimes & M & \rightarrow & \mathbb{C}[t,t^{-1}]\\
& \downarrow &  & \downarrow &  & \\
K_J: & N & \otimes & N & \rightarrow & \mathbb{C}[t,t^{-1}]\\
& \downarrow &  & \downarrow &  & \\
\times : & \mathbb{Q}_{S}^H & \otimes & \mathbb{Q}_{S}^H & \rightarrow & \mathbb{C}[t,t^{-1}]\\ 
\end{array}$
\end{center}

where the map in the first line is the polarization of the mixed Hodge module $M$, the third map is the product map and the middle one is a descent of the map $S$ on the Neron model. 

\begin{example}
We give an example of a degenerating Neron model for Jacobian bundles, to provide some picture of the construction, and leave more details for further studies. 
The example is taken from \cite{SCH2}, page 52 and belongs to M. Saito. Lets remark that there exists different notions of extensions for Jacobian bundles. In this example we only describe its construction over a Deligne extension. The minimal extension process is left to the reader as above. Let $H_{\mathbb{Z}}=\mathbb{Z}^4$, with $\mathbb{R}$-split Hodge structure given by $I^{1,-1}\oplus I^{-1,1} \oplus I^{0,2} \oplus I^{2,0}$, and $S$ be given by

\[ Q=\left( 
\begin{array}{cccc}
 0 & 0  & 1 & 0 \\
 0 & 0  &  0 &  1\\
1 & 0  &  0 &  0 \\
 0 &  1  & 0 & 0 
\end{array} \right)\]

\noindent
and nilpotent operator

\[ N_1=N_2=\left( 
\begin{array}{cccc}
 0 & 0  &  1 & 0 \\
 0 & 0  &  0 &  1  \\
 0 & 0  &  0 &  0  \\
 0 & 0  &  0 &  0  
\end{array} \right).\]

\noindent
Let $\omega \in \mathbb{C}$ have $Im(\omega) \ne 0$. If the mixed Hodge structure be split over $\mathbb{Z}$, we may set

\[ I^{1,-1}= \mathbb{C} \left( 
\begin{array}{c}
 0 \\
 0 \\
 1 \\
 \omega   
\end{array} \right), \qquad I^{1,-1}= \mathbb{C} \left( 
\begin{array}{c}
 0 \\
 0 \\
 1 \\
 \bar{\omega}   
\end{array} \right), \qquad I^{1,-1}= \mathbb{C} \left( 
\begin{array}{c}
 1 \\
 \omega \\
 0 \\
 0   
\end{array} \right), \qquad I^{1,-1}= \mathbb{C} \left( 
\begin{array}{c}
 1 \\
 \bar{\omega} \\
 0 \\
 0   
\end{array} \right). \]

\noindent
These data define an $\mathbb{R}$-split nilpotent orbit on $(\Delta^*)^2$, by the rule $(z_1,z_2) \to e^{z_1N_1+z_2N_2}F$, where $F$ is given by $I^{p,q}$. it is a pull back of a nilpotent orbit on $(\Delta^*)^2$ by the map $(z_1,z_2) \mapsto z_1z_2$. $F^0$ on the Deligne extension is spanned by

\[ e_0=  \left( 
\begin{array}{c}
 0 \\
 0 \\
 1 \\
 \omega   
\end{array} \right), \qquad e_1= \dfrac{1}{s_1} \left( 
\begin{array}{c}
 1 \\
 \omega \\
 0 \\
 0  
\end{array} \right), \qquad e_1= \dfrac{1}{s_2}\left( 
\begin{array}{c}
 1 \\
 \omega \\
 0 \\
 0   
\end{array} \right). 
\] 

\noindent
It has a presentation as

\[ \mathcal{O} \stackrel{\left( 
\begin{array}{c}
 0\\
 -s_1 \\
 s_2   
\end{array} \right) }{\longrightarrow} \mathcal{O}^3 \to F^0 \to 0. 
\]

Thus, $F^0$ is the subset of $\Delta^2 \times \mathbb{C}^3$ given by the equation $s_1v_1=s_2v_2$, using the coordinate $(s_1,s_2,v_1,v_2,v_3)$. Therefore, the Jacobian bundle $T$, is a bundle of rank 2 outsite the origin and has fiber $\mathbb{C}^3$ over $0$. Lets look at the embedding of $T_{\mathbb{Z}}$. If $h \in \mathbb{Z}^4$ is any integral vector, one has

\begin{center}
$S(e_0,e^{z_1N_1+z_2N_2}h)=(z_1+z_2)(h_3+h_4\omega)-(h_1+h_2\omega)$\\[0.2cm]
$S(e_j,e^{z_1N_1+z_2N_2}h)=-(h_3+h_4\omega)/s_j, \qquad j=1,2$.
\end{center}

Then the closure of $T_{\mathbb{Z}}$ is given by 

\[
(e^{2 \pi iz_1},e^{2\pi i z_2}, (z_1+z_2)(h_3+h_4\omega)-(h_1+h_2\omega), -\dfrac{(h_3+h_4\omega)}{e^{2 \pi iz_1}},-\dfrac{(h_3+h_4\omega)}{e^{2 \pi iz_2}}).
\]

Then the Jacobian bundle over $(\Delta^*)^2$ consists of usual intermediate Jacobians. However over $0$ is $J \times \mathbb{C}^2$, and over remaining points with $s_1s_2=0$ is $J \times \mathbb{C}$, where $J=\mathbb{C}/\mathbb{Z}+\mathbb{Z}\omega$ (see the reference for more details).
 
\end{example}

\nocite{*}
\bibliographystyle{plain}

\end{document}